\documentclass[11pt]{amsart}
\usepackage{geometry}                
\usepackage{amssymb,graphics,epsfig,latexsym}
\usepackage[centertags]{amsmath}
\usepackage{amsfonts}
\usepackage{ccaption}
\usepackage{lettrine}
\usepackage[all]{xy}
\usepackage[numbers,sort&compress]{natbib}
\bibpunct{[}{]}{;}{n}{,}{,}

\def\RR{\mathbb{R}}
\newcommand{\dtq}{\dot{q}}
\newcommand{\ext}{{\rm ext}}

\def\ben{\begin{equation}}
\def\een{\end{equation}}
\def\be{\begin{equation*}}
\def\ee{\end{equation*}}
\def\bi{\begin{itemize}}
\def\ei{\end{itemize}}
\def\ba{\begin{array}}
\def\ea{\end{array}}

\newtheorem{theorem}{Theorem}

\newtheorem{lemma}{Lemma}

\begin{document}

\title{Prolongation-Collocation Variational Integrators}
\author{Melvin Leok}
\address{Department of Mathematics, University of California, San Diego.}
\email{mleok@math.ucsd.edu}
\author{Tatiana Shingel}
\address{Department of Mathematics, University of California, San Diego.}
\email{tshingel@math.ucsd.edu}

\begin{abstract}
We introduce a novel technique for constructing higher-order variational integrators for Hamiltonian systems of ODEs.
In particular, we are concerned with generating globally smooth approximations to solutions of a Hamiltonian system.
Our construction of the discrete Lagrangian adopts Hermite interpolation polynomials and the Euler--Maclaurin quadrature formula, and involves applying collocation to the Euler--Lagrange equation and its prolongation.
Considerable attention is devoted to the order analysis of the resulting variational integrators in terms of approximation properties of the Hermite polynomials and quadrature errors. A performance comparison is presented on a selection of these integrators.
\end{abstract}
\date{}
\maketitle

\section{Introduction}

One of the major themes in geometric integration is symplectic methods for solving Hamiltonian systems of ordinary differential
equations. Viewed as maps, such methods preserve the symplectic two-form underlying the dynamical evolution of the system.
Variational integrators are an important class of symplectic integrators, which arise from discretizing Hamilton's principle.
We refer the reader to \cite{marsden&west} for a detailed discussion of the background theory of such methods.
Variational integrators are automatically symplectic and momentum preserving. Moreover, they exhibit good energy
behavior for exponentially long times.

The construction of variational integrators combines the techniques from approximation theory and numerical
quadrature and is related to the Galerkin approach of converting a differential operator equation into a discrete system.
Galerkin methods are semi-analytic in the sense that the discrete solution is described by an element of a
finite-dimensional function space, which provides an analytic expression for the numerical solution. Our goal is to develop
variational integrators that would lead to globally smooth approximations of the solution. To pursue this goal, we adopt
the space of piecewise Hermite polynomials in the Galerkin construction. It is a well-known
result in approximation theory that Hermite polynomials allow for higher-order approximation
of smooth functions at relatively low cost. Recall that Lagrange polynomials are a common choice in the construction of higher-order variational integrators.
However, the computation of an interpolating polynomial in Lagrange form becomes unstable when the degree of the polynomial is high. Unlike Lagrange polynomials, higher-degree Hermite polynomials produce accurate and computationally
stable results.

Furthermore, the Galerkin construction based on piecewise Hermite polynomial interpolation allows us to adequately address the
question of variational order error analysis. The order error analysis of variational integrators relies on determining the order with which a discrete Lagrangian $L_d:Q\times Q\rightarrow \mathbb{R}$ approximates the \textit{exact discrete Lagrangian},
\begin{equation}
\label{exact_classical}
L_d^E(q_0, q_1;h) = \int_0^h L(q_{01}(t), \dot q_{01}(t)) dt,
\end{equation}
where $q_{01}(t)$ is a solution curve of the Euler--Lagrange equation that satisfies the boundary conditions $q_{01}(0)=q_0$, $q_{01}(h)=q_1$. The standard way of performing the error analysis is by comparing the Taylor expansions of the exact discrete Lagrangian and the discrete Lagrangian. Consequently, the result critically depends on the extent to which the discrete trajectory is able to approximate the higher-derivatives of the exact Euler--Lagrange solution curve. To enable a robust variational order error analysis for the Galerkin variational integrators, we propose a novel application of the collocation approach in the setting of discrete Lagrangian mechanics, which involves prolongations of the Euler--Lagrange vector field.
The proposed approach has the advantage that one is able to prove \emph{optimal} rates of convergence of the associated variational integrators as long as sufficiently accurate quadrature formulas are used. The proof crucially depends on the preliminary estimates on the higher-derivative approximation properties of prolongation-collocation curves.

\subsection{Outline of the Paper}
We present a brief review of discrete variational mechanics and variational integrators in Section~\ref{sec:variational_integrators}, and the Euler--Maclaurin quadrature formula in Section~\ref{sec:quadrature}. In Section~\ref{section_method}, we introduce prolongation-collocation variational integrators, and perform variational order error analysis in Section~\ref{section_VOCal}. In Section~\ref{section_examples}, we present a few numerical examples, and present some conclusions and future directions in Section~\ref{sec:conclusions}.

\section{Variational Integrators}\label{sec:variational_integrators}

Let $Q$ be the configuration manifold of a mechanical system, with generalized coordinates $q$.
Consider the \textit{Lagrangian} $L:TQ\rightarrow\RR$, where $TQ$ is the tangent bundle of
the configuration space $Q$.  The tangent bundle has local coordinates $(q,v)$. Further, let $\mathcal{C}(Q)=\mathcal{C}([0,T],Q)$ denote
the space of smooth trajectories $q:[0,T]\rightarrow Q$ in the configuration manifold $Q$.
The \emph{action integral} $S:\mathcal{C}(Q)\rightarrow \RR$ is defined as
\be
S(q)=\int_{0}^{T}L(q(t),\dtq(t))dt.
\ee
The variational principle known as \emph{Hamilton's Principle} states that
$$
\delta S=0, \mbox{ for } \delta q(0)=\delta q(T)=0,
$$
which yields the \emph{Euler--Lagrange equations}
\ben
\label{ELeq}
\frac{\partial L}{\partial q}(q,\dtq)-\frac{d}{dt}\left(\frac{\partial L}{\partial \dtq}(q,\dtq)\right)=0.
\een
It is possible to rewrite the equations \eqref{ELeq} in terms of the generalized coordinates and momenta $(q,p)$ on
the cotangent bundle $T^{*}Q$ (\emph{phase space}). For this, we introduce
the \emph{Legendre transformation} $\mathbb{F}L:TQ\rightarrow T^{*}Q$, defined by
\be
\mathbb{F}L:(q,\dtq)\mapsto \left(q,\frac{\partial L}{\partial \dtq}\right).
\ee
The \emph{Hamiltonian} $H:T^{*}Q\rightarrow\RR$ is given by
\be
H(q,p)=\left.p\cdot\dtq-L(q,\dtq)\right|_{p=\frac{\partial L}{\partial\dot q}}.
\ee
One can show that equations \eqref{ELeq} are equivalent to \emph{Hamilton's equations} (Theorem 1.3, p. 182 in \cite{hairerbook}),
\ben
\label{HAMeq}
\dot{p}=-\frac{\partial H}{\partial q}(p,q),\qquad \dtq=\frac{\partial H}{\partial p}(p,q).
\een
In the case of variational integrators, instead of discretizing the Euler--Lagrange equations \eqref{ELeq}, one
discretizes Hamilton's principle. That is, one discretizes the action by introducing a \emph{discrete Lagrangian} and replaces
the action integral by an \emph{action sum}, and applies the discrete version of the Hamilton's variational principle.
The major advantage of this approach is that the resulting numerical algorithm automatically preserves
the symplectic structure of the underlying dynamical system.

The discrete Lagrangian $L_{d}(q_{0},q_{1},h)$ is thought of as an
approximation of the action integral along the curve segment between the points $q_{0}\approx q(0)$ and $q_{1}\approx q(h)$.
Formally, this can be expressed as
\ben
\label{DL}
L_{d}(q_{0},q_{1},h)\approx\int_{0}^{h}L(q(t),\dtq(t))dt.
\een
We will neglect the $h$-dependence and simply write $L_{d}(q_{0},q_{1})$
when it is not essential for the exposition of the material. Given the discrete sequence of times $\{t_{k}=hk\,|\, k=0,\ldots,N\}$,
$h=T/N$, a discrete curve in $Q$ is denoted by $\{q_{k}\}_{k=0}^{N}$, where $q_k\approx q(t_{k})$.
The discrete action sum is a function that maps the discrete trajectories $\{q_{k}\}_{k=0}^{N}$ to $\RR$, and is given by
$$
S_{d}(\{q_{k}\}_{k=0}^{N})=\sum_{k=0}^{N-1}L_{d}(q_{k},q_{k+1}).
$$
The \emph{discrete Hamilton's principle} requires the discrete action to be stationary with respect to variations vanishing
at $k = 0$ and $k = N$. From this, one derives the discrete version of the Euler--Lagrange equations, which
are known as the \emph{discrete Euler--Lagrange equations},
\ben
\label{DEL_eqn}
D_{2}L_{d}(q_{k-1},q_{k})+D_{1}L_{d}(q_{k},q_{k+1})=0,
\een
where $k=1,2,\ldots,N-1$. These equations implicitly define the one-step \emph{discrete Lagrangian map} $F_{L_{d}}:Q\times Q\rightarrow Q\times Q$.
The discrete Legendre transforms $\mathbb{F}^{\pm}L_{d}:Q\times Q\rightarrow \RR$
are defined by
\be
\begin{split}
&\mathbb{F}^{-}L_{d}:(q_{0},q_{1})\mapsto (q_{0},-D_{1}L_{d}(q_{0},q_{1})),\\&
\mathbb{F}^{+}L_{d}:(q_{0},q_{1})\mapsto (q_{1},D_{2}L_{d}(q_{0},q_{1})).
\end{split}
\ee
Pushing the discrete Lagrangian map $F_{L_{d}}$ forward to $T^{*}Q$ with the discrete Legendre transforms gives the \emph{discrete Hamiltonian map}
$\tilde{F}_{L_{d}}:T^{*}Q\rightarrow T^{*}Q$ by $\tilde{F}_{L_{d}}=\mathbb{F}^{\pm}L_{d}\circ F_{L_{d}}\circ (\mathbb{F}^{\pm}L_{d})^{-1}$.
The fact that the definitions of $\tilde{F}_{L_{d}}$ are equivalent for the $+$ and $-$ case is implied by the following commutative diagram
(Theorem 1.5.2 in \cite{marsden&west})
\ben
\label{diagram1}
\vcenter{
\xymatrix@!0@R=0.75in@C=.75in{ & (q_{0},q_{1})\ar@{|->}[dl]_{\mathbb{F}^{-}L_{d}} \ar@{|->}[dr]^{\mathbb{F}^{+}L_{d}} \ar@{|->}[rr]^{F_{L_d}} & & (q_1,q_2) \ar@{|->}[dl]^{\mathbb{F}^{-}L_{d}} \ar@{|->}[dr]^{\mathbb{F}^{+}L_{d}} &
\\ (q_{0},p_{0}) \ar@{|->}[rr]_{\tilde{F}_{L_d}} & &(q_{1},p_{1}) \ar@{|->}[rr]_{\tilde{F}_{L_d}} &  &(q_2,p_2)
}}
\een
In coordinates, $\tilde{F}_{L_{d}}:(q_{0},p_{0})\mapsto (q_{1},p_{1})$, where
\ben
\label{PM_eq}
p_{0}=-D_{1}L_{d}(q_{0},q_{1}),\qquad   p_{1}=D_{2}L_{d}(q_{0},q_{1}).
\een
A numerical quadrature can be used to approximate the integral in \eqref{DL}. However, the functional form of the solution curve $q(t)$ is
required when applying a quadrature rule and it is, in general, unknown. In practice, one can
choose an interpolating function on the interval $[0,h]$ passing through $q_{0}$, $q_{1}$. Then, a quadrature rule can be applied
to the integral of the Lagrangian evaluated along the interpolating function. This approach fits the general framework
of Galerkin integration methods. In more detail, for the construction of Galerkin Lagrangian variational integrators,
one replaces the path space $\mathcal{C}([0,T],Q)$, which is an infinite-dimensional
function space, with a finite-dimensional function space, $\mathcal{C}^{s}([0,T],Q)$.
Commonly, one uses polynomial approximations to the trajectories, letting
\be
\mathcal{C}^{s}([0,h],Q)=\{q\in \mathcal{C}([0,h],Q)\, |\, q\mbox{ is a polynomial of degree }\leq s\}.
\ee
An approximate action $\mathcal{S}(q):\mathcal{C}^{s}([0,h],Q)\rightarrow\RR$ is
\begin{equation*}
\mathcal{S}(q)=h\sum_{i=1}^{s}b_{i}L(q(c_{i} h),\dot{q}(c_{i}h)),
\end{equation*}
where $c_{i}\in[0,1]$ are quadrature points, $b_{i}$ are quadrature weights, $i=1,\ldots,s$.
The Galerkin discrete Lagrangian is
\begin{equation}
\label{GalDL}
L_{d}(q_{0},q_{1})=\underset{q\in\mathcal{C}^{s}([0,h],Q) }{\mathrm{ext}}\mathcal{S}(q).
\end{equation}
In particular, for higher-order methods one takes $q\in\mathcal{C}^{s}([0,h],Q)$ in the form
\begin{equation*}
q(\tau h;q_{0}^{\nu},h)=\sum_{\nu=0}^{s}q_{0}^{\nu}\tilde{l}_{\nu,s}(\tau),
\end{equation*}
where $q_{0}^{\nu}$, $\nu=1,\ldots,s-1$, are the \emph{internal stages}, $\tilde{l}_{\nu,s}(\tau)$ are the \emph{Lagrange basis polynomials}
 of degree $s$ defined on the interval $[0,1]$. Then, the integration scheme $(q_{0},p_{0})\mapsto (q_{1},p_{1})$ is given by
\begin{equation*}
\begin{aligned}
-p_{0}&=h\sum_{i=1}^{s}b_{i}\left[\frac{\partial L}{\partial q}(c_{i}h)\tilde{l}_{0,s}(c_{i})+\frac{1}{h}\frac{\partial L}{\partial \dot{q}}(c_{i} h)\dot{\tilde{l}}_{0,s}(c_{i})\right],\\
0&=h\sum_{i=1}^{s}b_{i}\left[\frac{\partial L}{\partial q}(c_{i}h)\tilde{l}_{\nu,s}(c_{i})+\frac{1}{h}\frac{\partial L}{\partial \dot{q}}(c_{i} h)\dot{\tilde{l}}_{\nu,s}(c_{i})\right], \nu=\overline{1,s-1}\\
p_{1}& = h\sum_{i=1}^{s}b_{i}\left[\frac{\partial L}{\partial q}(c_{i}h)\tilde{l}_{s,s}(c_{i})+\frac{1}{h}\frac{\partial L}{\partial \dot{q}}(c_{i} h)\dot{\tilde{l}}_{s,s}(c_{i})\right]
\end{aligned}
\end{equation*}
It has been established (see \cite{marsden&west,hairerbook}) that the above scheme is equivalent to a
symplectic partitioned Runge-Kutta method. Note that it yields a discrete solution that is, in general, only piecewise regular. In what follows, we introduce a
construction of higher-order variational integrators with improved regularity across nodal times, which has as its natural variables the position, and its derivatives at only the nodal times, without the use of internal stages.

\section{Quadrature}\label{sec:quadrature}
Here, we present a quadrature formula that we will use in our construction of a high-order discrete Lagrangian. The advantage of this particular rule is that it only involves function evaluations at the endpoints of the interval. When fast adaptive treecodes are used in conjunction with automatic differentiation techniques~\cite{rall1981}, it is more efficient to obtain higher-order approximations using higher-derivative information at the endpoints, rather than evaluating the integrand at a number of internal stages.

\begin{theorem}\label{EMF} \textbf{(Euler--Maclaurin quadrature formula)}\cite{abramowitz}
If $f$ is sufficiently differentiable on $(a,b)$, then for any $m>0$
\begin{multline*}
\int_{a}^{b}f(x)dx = \frac{\theta}{2}\left[f(a)+2\sum_{k=1}^{N-1}f(a+k\theta)+f(b)\right]\\
-\sum_{l=1}^{m}\frac{B_{2l}}{(2l)!}\theta^{2l}\left(f^{(2l-1)}(b)-f^{(2l-1)}(a)\right)
-\frac{B_{2m+2}}{(2m+2)!}N \theta^{2m+3}f^{(2m+2)}(\xi)
\end{multline*}
where $B_k$ are the Bernoulli numbers, $\theta=(b-a)/N$ and $\xi\in(a,b)$.
\end{theorem}

Let us apply Theorem~\ref{EMF} to approximate an integral $\int_{0}^{h}f(x)dx$ in the simplest case when $N=1$.
It is easy to see that we obtain the following quadrature rule
\ben
\label{EMQ}
K(f)=\frac{h}{2}\left[f(0)+f(h)\right]-\sum_{l=1}^{m}\frac{B_{2l}}{(2l)!}h^{2l}\left(f^{(2l-1)}(h)-f^{(2l-1)}(0)\right),
\een
and the error of approximation is $\mathcal{O}(h^{2m+3})$.

\section{The Prolongation-Collocation Method}\label{section_method}
In this section, we explain the construction of the discrete Lagrangian based on Hermite interpolation and the Euler--Maclaurin
quadrature formula.
\subsection*{Motivation for Prolongation-Collocation Approach}
The variational characterization of the exact discrete Lagrangian \eqref{exact_classical} naturally leads to the variational Galerkin discrete Lagrangian \eqref{GalDL}, where the infinite-dimensional function space $\mathcal{C}([0,h],Q)$ is replaced by a finite-dimensional subspace, and the integral is approximated by a quadrature formula. While this leads to a computable discrete Lagrangian, one does not necessarily obtain an optimally accurate discrete Lagrangian whose variational order is related to the best approximation properties of the chosen finite-dimensional function space. In particular, one finds that the variational Galerkin extremal curves do not necessarily approximate the higher-derivatives of the Euler--Lagrange solution curves with adequate accuracy.

In retrospect, the fact that the variational Galerkin approach does not readily lead to computable discrete Lagrangians with provable approximation properties is not too surprising.  By construction, variational Galerkin discrete Lagrangians associated with a sequence of finite-dimensional function spaces involve extremizers of a sequence of functionals. Since the sequence of finite-dimensional function spaces converges to $\mathcal{C}([0,h],Q)$, the sequence of functionals converges to the functional that appears in the variational characterization of the exact discrete Lagrangian. However, it is unclear that the sequence of extremizers converges to the extremizer of the limiting functional, since that corresponds to $\Gamma$-convergence~\cite{Da1993} of the sequence of functionals. The issue of optimal rates of convergence of the computable discrete Lagrangians involves establishing rates of convergence of extremizers in terms of approximation rates of the finite-dimensional function spaces, which is an even more complicated process.

As an alternative, we adopt the characterization of the exact discrete Lagrangian in terms of the Euler--Lagrange solution curve, and construct  a discrete curve which approximates higher-derivatives of the Euler--Lagrange solution curve to an adequate level of accuracy. The latter is explored in detail in Section~\ref{section_VOCal}.

\subsection*{Hermite Interpolation and Prolongation-Collocation}
We commence by replacing $q(t)$ in \eqref{DL} by its Hermite interpolant which is obtained by
constructing a polynomial $q_{d}(t)$  such that values of $q(t)$ and any number of its derivatives at given points are fitted by
the corresponding function values and derivatives of $q_{d}(t)$. In this paper we are concerned with fitting
function values of $q(t)$ and its derivatives at the end-points of the interval $[0,h]$.
Consequently, a so-called two-point Hermite interpolant $q_{d}(t)$ of degree $d=2n-1$ can be used, which has the form
\begin{align}
\label{HermitePol}
q_{d}(t)&=\sum_{j=0}^{n-1}\left(q^{(j)}(0)H_{n,j}(t)+(-1)^{j}q^{(j)}(h)H_{n,j}(h-t)\right),
\intertext{where}\label{basis_pol}
H_{n,j}(t)&=\frac{t^{j}}{j!}(1-t/h)^{n}\sum_{s=0}^{n-j-1}\left( \ba{c} n+s-1 \\ s \ea \right)(t/h)^{s}
\end{align}
are the Hermite basis functions. Note that for $n=1$, the interpolant is a straight line joining $q(0)$ and $q(h)$.
By choosing one of the simple quadrature rules to discretize the integral in \eqref{DL} (e.g., the midpoint rule or trapezoidal rule),
one obtains a class of well-known integrators which are at most second-order (see \cite{marsden&west}).
Therefore, the first nontrivial case of interest is $n=2$, where we assume that the position and velocity data
at the end points are available. From now on, we only consider $n\geq 2$ when applying the Hermite interpolation formula.
The detailed derivation of \eqref{HermitePol} can be found, for example, in \cite{davis}. By construction,
$$
q_{d}^{(r)}(0)=q^{(r)}(0), \qquad q_{d}^{(r)}(h)=q^{(r)}(h), \qquad r=0,1,\ldots n-1.
$$
Except for the step-size $h$, the discrete Lagrangian $L_{d}(q_0,q_1,h)$ should only depend on $q_{0}\approx q(0)$, $q_{1}\approx q(h)$.
Therefore, letting $q_{d}(0)=q_{0}$ and $q_{d}(h)=q_1$, we need to approximate the higher-order derivatives of $q(t)$ by
expressions that only depend on $q_{0},q_{1}$. One natural approach, which is often found in the literature,
is to use finite differences. In this work, we propose to apply
the idea of \emph{collocation} in conjunction with the Euler--Lagrange equations \eqref{ELeq}.
The benefits of this approach will be exemplified later when discussing the variational error analysis of the
proposed class of numerical integrators (see Section~\ref{section_VOCal}).

The collocation approach~\cite{HaNoWa1993} is well-known
in the theory of initial and boundary value problems for ODEs \cite{costabile,iserlesbook}.
Roughly speaking, the technique consists of determining the unknown parameters of a parameterized curve by requiring
$q_{d}(t)$ to satisfy the ODE at a given set of points (collocation points).
To define $q_{d}(t)$ uniquely, one sets the number of collocation points to be equal to the number of the available degrees of freedom.
In our approach we use the method of collocation in a slightly unusual manner.
In particular, since the parameters in \eqref{HermitePol} correspond to the derivatives of the solution curve $q(t)$
at the end points of the interval $[0,h]$, we are going to use $t=0$ and $t=h$ as collocation points for the Euler--Lagrange
equations \eqref{ELeq} and consider the \emph{prolongation} \cite{Ol1993} of the Euler--Lagrange equations in order to generate a sufficient number of conditions. In other words, we increase the number
of equations under consideration (not the number of collocation points) to match the number of degrees of freedom.

For example, consider the case of the quintic Hermite interpolation, i.e., set $n=3$ in \eqref{HermitePol}.
For separable Lagrangians of the form
$L(q,\dtq)=\frac{1}{2}m\dtq^2-V(q)$, where $m$ is the mass and $V(q)$ is the potential energy term, the Euler--Lagrange
equations \eqref{ELeq} become a second-order ODE of the form
\begin{align*}
\ddot{q}(t)&=f(q(t)),
\intertext{and its first-order prolongation can be expressed as}
q^{(3)}(t)&=f'(q(t))\dot{q}(t).
\end{align*}
We set the boundary conditions $q_{d}(0)=q_{0}$ and $q_{d}(h)=q_{1}$ and the collocation conditions
\begin{alignat*}{2}
\ddot{q}_{d}(0)&=f(q_{d}(0)),&\qquad q^{(3)}_{d}(0)&=f'(q_{d}(0))\dot{q}_{d}(0),\\
\ddot{q}_{d}(h)&=f(q_{d}(h)),&\qquad q^{(3)}_{d}(h)&=f'(q_{d}(h))\dot{q}_{d}(h).
\end{alignat*}
The above conditions constitute the system of six equations, which uniquely determines the fifth-degree polynomial $q_{d}(t)$ in the form of \eqref{HermitePol}.

In general, for the Hermite polynomial of degree $2n-1$, one would need to differentiate the Euler--Lagrange equation $n-2$ times, thus
deriving a system of $n-1$ equations for $\ddot{q}_{d}(t), q^{(3)}_{d}(t), \ldots, q^{(n)}_{d}(t)$.
Evaluated at $0$ and $h$, these (together with the Euler--Lagrange equations) give $2n-2$ collocation equations, which together with boundary conditions ($q_d(0)=q_0, q_d(h)=q_1$)
constitute a sufficient number of conditions to determine the interpolant $q_{d}(t)$ uniquely. Note that for
large $n$, the system of collocation conditions becomes nonlinear.
Since the second and higher-order derivatives $q_{d}^{(j)}(0), q_{d}^{(j)}(h)$ are given explicitly, the system can be
recursively reduced to two implicit equations involving $\dtq_{d}(0), \dtq_{d}(h)$.
In this case, one would need to make use of a nonlinear root solver, such as the Newton--Raphson method, to determine $\dtq_{d}(0), \dtq_{d}(h)$.

Further, in order to discretize the integral in \eqref{DL}, we apply the Euler--Maclaurin quadrature formula \eqref{EMQ}.
Recall that the formula involves derivatives of the integrand, in our case the Lagrangian $L$, with
respect to the independent variable evaluated at the end-points of the
integration interval. The latter, however, does not require the extensive computations typically associated with the interpolating polynomial
due to the use of the Hermite interpolation formula and
the collocation idea explained above. In more detail, we write
\begin{multline*}
\int_{0}^{h}L(q_{d}(t),\dtq_{d}(t))dt\approx\frac{h}{2}(L(q_{d}(0),\dtq_{d}(0))+L(q_{d}(h),\dtq_{d}(h)))\\
-\sum_{l=1}^{m}\frac{B_{2l}}{(2l)!}h^{2l}\left(\frac{d^{2l-1}}{dt^{2l-1}}L(q_{d}(t),\dtq_{d}(t))\bigg|_{t=h}-
\frac{d^{2l-1}}{dt^{2l-1}}L(q_{d}(t),\dtq_{d}(t))\bigg|_{t=0}\right).
\end{multline*}
Provided that the degree of the interpolating polynomial is $2n-1$, we choose $m\leq \lfloor n/2\rfloor$, where the brackets denote
the greatest integer lower bound for $n/2$. So for even $n$, $\lfloor n/2\rfloor=n/2$ and for odd $n$, $\lfloor n/2\rfloor=(n-1)/2$.
This choice of $m$ is justified by observing that the expressions for
$$
\frac{d^{2l-1}}{dt^{2l-1}}L(q_{d}(t),\dtq_{d}(t))\bigg|_{t=\tau},\quad l=1,2,\ldots,m,\quad \tau=0,h,
$$
include the derivatives of $q_{d}(t)$ up to order $2\lfloor n/2\rfloor-1+1\leq n$, which satisfy the
corresponding collocation conditions.

\subsection*{Prolongation-Collocation Discrete Lagrangian}

The Prolongation-Collocation discrete Lagrangian is defined as follows,
\begin{multline}
\label{DLEM}
L_{d}(q_{0},q_{1},h)=\frac{h}{2}(L(q_{d}(0),\dtq_{d}(0))+L(q_{d}(h),\dtq_{d}(h)))\\
-\sum_{l=1}^{\lfloor n/2\rfloor}\frac{B_{2l}}{(2l)!}h^{2l}\left(\frac{d^{2l-1}}{dt^{2l-1}}L(q_{d}(t),\dtq_{d}(t))\bigg|_{t=h}-
\frac{d^{2l-1}}{dt^{2l-1}}L(q_{d}(t),\dtq_{d}(t))\bigg|_{t=0}\right),
\end{multline}
where $q_d(t)\in \mathcal{C}^s(Q)$ is determined by the boundary and prolongation-collocation conditions,
\begin{alignat}{2}
\label{DLsys}
 q_{d}(0) &=q_{0} &\qquad q_{d}(h)&=q_{1},\notag \\
\ddot{q}_{d}(0)&=f(q_{0})&\qquad \ddot{q}_{d}(h)&=f(q_{1}),\notag\\
q_{d}^{(3)}(0)&=f'(q_{0})\dtq_{d}(0)&\qquad q_{d}^{(3)}(h)&=f'(q_{1})\dtq_{d}(h),\\
&\hspace{1.25ex}\vdots &\qquad & \hspace{1.25ex}\vdots\notag\\
q^{(n)}_{d}(0)& =\frac{d^{n}}{dt^{n}}f(q_{d}(t))\bigg|_{t=0} &\qquad q^{(n)}_{d}(h)& =\frac{d^{n}}{dt^{n}}f(q_{d}(t))\bigg|_{t=h}\notag
\end{alignat}
One can include fewer than $\lfloor n/2\rfloor$ terms in the summation in \eqref{DLEM}. However, this will have an impact on the variational
order of the corresponding integrator as will be further discussed in Section~\ref{section_VOCal}. The system of equations \eqref{DLsys} completely
defines the discrete Lagrangian \eqref{DLEM}. Note that we used the second-order ODE $\ddot{q}(t)=f(q(t))$ as a prototype
of the Euler--Lagrange equations for simplicity of notation only. The same idea applies for any smooth, not necessarily separable, Lagrangian
function and the corresponding Euler--Lagrange equations.

Given the initial conditions $(q_{0},p_{0})$, the variational integrator has the form
\ben
\label{one-stepVI}
\begin{aligned}
p_{k} &=-D_{1}L_{d}(q_{k},q_{k+1}),\\
p_{k+1} &=D_{2}L_{d}(q_{k},q_{k+1}), \qquad k=0,1,\ldots,
\end{aligned}
\een
and defines a one-step map $(q_{k},p_{k})\mapsto(q_{k+1},p_{k+1})$. Generally, the equation $p_{k}=-D_{1}L_{d}(q_{k},q_{k+1})$ together
with the system of collocation conditions \eqref{DLsys} can be reduced to a system of implicit equations with respect to $q_{k+1}, \dtq_{d}(t_{k}),
\dtq_{d}(t_{k+1})$. As soon as the solution is obtained via some appropriate nonlinear root-finding method, it is inserted into the equation
$p_{k+1}=D_{2}L_{d}(q_{k},q_{k+1})$.

When the Lagrangian has a relatively simple form, it makes sense to compute the expression for the discrete Lagrangian
\eqref{DLEM} symbolically, which can be done using the symbolic module in \textsc{Matlab} or symbolic software such as \textsc{Mathematica} or \textsc{Maple}.
Having computed \emph{a priori} closed-form expressions for the right-hand side in \eqref{one-stepVI}, makes the implementation of the integrator particularly simple and fast. The reader is referred to Section~\ref{section_examples} for some numerical examples.

\section{Variational Order Calculation}\label{section_VOCal}
The construction of variational integrators in the Galerkin framework naturally leads to the question of how
it can be reconciled with the results from approximation theory of function spaces and
numerical analysis of quadrature schemes. In particular, our goal is to explore the way the quantitative characteristics of the
approximation errors enter the calculation of the convergence order of the respective integrators.
Variational error analysis provides the right framework to pursue this goal.

The variational error analysis introduced in \cite{marsden&west}, and refined in \cite{PaCu2009}, is based on the idea that rather than considering how closely the numerical trajectory matches the exact flow,
one can consider how the discrete Lagrangian approximates the exact discrete Lagrangian \eqref{exact_classical}
which generates the exact flow map of the Euler--Lagrange equations. In other words, we are looking at the approximation error in
\be
L_{d}(q(0),q(h))\approx\underset{\underset{q(0)=q_{0},\,q(h)=q_{1}}{q\in \mathcal{C}([0,h],Q)}}{\ext}\int_{0}^{h}L(q(t),\dtq(t))dt=L_{d}^{E}(q(0),q(h),h).
\ee
We say that a given discrete Lagrangian is of \emph{order $r$} if there exist an open subset $U_{v}\subset TQ$ with
compact closure and constants $C_{v}$ and $h_{v}>0$ so that
\ben
\label{order_def}
\|L_{d}(q(0),q(h),h)-L_{d}^{E}(q(0),q(h),h)\|\leq C_{v}h^{r+1}
\een
for all solutions $q(t)$ of the Euler--Lagrange equations with initial condition $(q(0),\dtq(0))\in U_{v}$ and for all $h\leq h_{v}$.
In \cite{marsden&west}, the authors prove the equivalence of \eqref{order_def} (cf. Theorem~2.3.1 in \cite{marsden&west}) to:
\begin{enumerate}
\renewcommand{\labelenumi}{(\roman{enumi})}
 \item the discrete Hamiltonian map $\tilde{F}_{L_{d}}$ being of order $r$;
 \item  the discrete Legendre transforms $\mathbb{F}^{\pm}L_{d}$ being of order $r$.
\end{enumerate}
In particular, the discrete Hamiltonian map is of order $r$ if
\ben
\label{DisHamil_order}
\|\tilde{F}_{L_{d}}(q(0),p(0),h)-\tilde{F}_{L_{d}^{E}}(q(0),p(0),h)\|\leq \tilde{C}_{v}h^{r+1},
\een
for all solutions $(q(t), p(t))$ of the Hamilton's equations with initial condition $(q(0),p(0))\in U_{w}\subset T^{*}Q$ and for all $h\leq h_{w}$.
The order of the discrete Legendre transforms is defined analogously. Recall from the diagram \eqref{diagram1} that
$$
\tilde{F}_{L_{d}}: (q_{0},p_{0})\mapsto (q_{1},p_{1}).
$$
By construction, $\tilde{F}_{L_{d}^{E}}(q(0),q(h),h)$
produces the values $(q(h),p(h))$ corresponding to the exact solution of the Hamiltonian system \eqref{HAMeq}, whereas $\tilde{F}_{L_{d}}(q(0),q(h),h)$ produces the approximate values $(q_{1},p_{1})$, $q_{1}\approx q(h)$, $p_{1}\approx p(h)$.
To summarize, the estimate \eqref{DisHamil_order} provides the local order of convergence of the discrete trajectory $(q_{k},p_{k})$
to the exact flow $(q(t),p(t))$ of the Hamiltonian vector field. This order is the same as the order to which the discrete Lagrangian approximates
the exact discrete Lagrangian, which we focus on.

Before we explore the inequality \eqref{order_def} for the Prolongation-Collocation discrete Lagrangian
discussed in Section~\ref{section_method}, we would like to establish the approximation error
of $q^{(j)}_{d}(t)$ in comparison to $q^{(j)}(t)$ for $j=1,2\ldots,n$, where $q(t)$ is the exact solution of
the Euler--Lagrange equation, and
$q_{d}(t)$ is the Hermite interpolating polynomial \eqref{HermitePol} of degree $2n-1$, constructed by letting $q_{d}(0)=q(0)$ and $q_{d}(h)=q(h)$, and
imposing the prolongation-collocation conditions discussed in Section~\ref{section_method} at the endpoints.
Note that this can be a difficult task in general, since the complexity of the collocation procedure escalates with
the degree of the Hermite polynomial. However, we are only interested in the approximation order at the end-points of the interval $[0,h]$,
in which case the analysis is straightforward.
\begin{lemma}\label{lem:HOD}
For $q(t)$, $q_{d}(t)$ as above, if $\dtq_{d}(\tau)=\dtq(\tau)+\mathcal{O}(h^p)$
for some $p>0$ and $\tau=0,h$, then
$$
q_{d}^{(j)}(\tau)=q^{(j)}(\tau)+\mathcal{O}(h^p),\quad j=3,\ldots,n.
$$
\end{lemma}
\begin{proof}
Note that $\ddot{q}(\tau)$ coincides with $\ddot{q}_{d}(\tau)$ by construction. Therefore,
we consider $j$ starting from $3$. As before, we restrict the proof to the case of a separable Lagrangian.
Indeed, since the Euler--Lagrange equation is equivalent to the second-order ODE
\ben
\label{newton}
\ddot{q}(t)=f(q(t)),
\een
it follows that for $\tau=0,h,$
\be
\begin{split}
q_{d}^{(3)}(\tau)-q^{(3)}(\tau)&=f'(q_{d}(\tau))\dtq_{d}(\tau)-f'(q(\tau))\dtq(\tau)\\&
=f'(q(\tau))\dtq_{d}(\tau)-f'(q(\tau))\dtq(\tau)=f'(q(\tau))(\dtq_{d}(\tau)-\dtq(\tau))=\mathcal{O}(h^p),
\end{split}
\ee
provided there exists a uniform bound for $f'$.
Consecutively differentiating \eqref{newton} and substituting the corresponding expressions for lower order
derivatives, one can see that $q^{(j)}(\tau)$ (and $q_{d}^{(j)}(\tau)$) can be represented as a polynomial in powers of
$\dtq(\tau)$ (resp. $\dtq_{d}(\tau)$) with coefficients which only depend on $q(\tau)=q_{d}(\tau)$,
\be
\begin{split}
q^{(4)}(\tau)&=f''(q(\tau))\dtq(\tau)^2+f'(q(\tau))f(q(\tau))\\
q^{(5)}(\tau)&=f^{(3)}(q(\tau))\dtq(\tau)^3+\dtq(\tau)(3f''(q(\tau))f(q(\tau))+f'(q(\tau))^2)\\
&\hspace{1.25ex}\vdots
\end{split}
\ee
As soon as $f$ has bounded higher-order derivatives, the above formulas imply that
the order of the approximation of $q^{(j)}(\tau)$ by $q_{d}^{(j)}(\tau)$ as a function of $h$ is
equal to the order of the approximation of  $\dtq(\tau)$ by $\dtq_{d}(\tau)$. We consider these expressions up to the $j=n$ case, where $n$ is determined
by the number of collocation equations that were used to compute $q_{d}$.
\end{proof}

In the next lemma, we determine the value of $p$ in the relation $\dtq_{d}(\tau)=\dtq(\tau)+\mathcal{O}(h^p)$
in terms of the degree of the polynomial $q_{d}(t)$.

\begin{lemma}\label{lem:FOD}
Consider a polynomial $q_{d}(t)$ of degree $d=2n-1$ given by the formula
$$
q_{d}(t)=\sum_{j=0}^{n-1}\left(a_{0j}H_{n,j}(t)+(-1)^{j}a_{1j}H_{n,j}(h-t)\right),
$$
where $H_{n,j}(t)$ are the basis polynomial functions \eqref{basis_pol}. By construction,
$$
q_{d}^{(j)}(0)=a_{0j}, \quad q_{d}^{(j)}(h)=a_{1j}, \quad j=0,\ldots n-1.
$$
We let $a_{00}=q(0)$, $a_{10}=q(h)$. The coefficients $a_{0j}$, $a_{1j}$, $j=1,2,\ldots, n-1$ are obtained from
the system of equations consisting of the Euler--Lagrange equation \eqref{ELeq} and its
prolongations. In particular, these are $n-1$ differential equations evaluated on $q_{d}(t)$ at $t=0$ and $t=h$.
Then for $\tau=0,h,$
\be
\dtq_{d}(\tau)=\dtq(\tau)+\mathcal{O}(h^{2n-1}).
\ee
\end{lemma}
\begin{proof}
Let $q(t)\in  \mathcal{C}^{2n}([0,h],Q)$  be the solution of the Euler--Lagrange equation \eqref{ELeq}
with boundary conditions $q(0)$ and $q(h)$. Then, $\ddot{q}(t)$ can be written in the form,
\be
\ddot{q}(t)=P_{2n-3}[\ddot{q}](t)+R_{n-1}[\ddot{q}](t),
\ee
where
\be
\begin{split}
P_{2n-3}[\ddot{q}](t)&=\sum_{j=0}^{n-2}\left(q^{(j+2)}(0)H_{n-1,j}(t)+(-1)^{j}q^{(j+2)}(h)H_{n-1,j}(h-t)\right),\\
R_{n-1}[\ddot{q}](t)&=\frac{q^{2n+2}(\xi)}{(2n-2)!}t^{n-1}(h-t)^{n-1},\quad \xi\in(0,h),
\end{split}
\ee
and $H_{n-1}(t)$ are the basis polynomials defined by \eqref{basis_pol}. See \cite{davis} for the proof of the above formula for sufficiently smooth functions. Note that $\ddot{q}_{d}(t)$ is a polynomial of degree $2n-3$ to which
the same formula can be applied. In the latter case, the remainder term is identically zero. Hence,
\be
\ddot{q}_{d}(t)=\sum_{j=0}^{n-2}\left(q_{d}^{(j+2)}(0)H_{n-1,j}(t)+(-1)^{j}q_{d}^{(j+2)}(h)H_{n-1,j}(h-t)\right).
\ee
Subtracting $\ddot{q}_{d}(t)$ from $\ddot{q}(t)$ gives
\begin{multline}\label{lem:formula1}
\ddot{q}(t)-\ddot{q}_{d}(t)\\=\sum_{j=1}^{n-2}\left((q^{(j+2)}(0)-q_{d}^{(j+2)}(0))H_{n-1,j}(t)+(-1)^{j}(q^{(j+2)}(h)-q_{d}^{(j+2)}(h))H_{n-1,j}(h-t)\right)\\
\quad+R_{n-1}[\ddot{q}](t).
\end{multline}
Next, we integrate the above expression from $0$ to $h$. The left-hand side of the equation becomes
\be
\int_{0}^{h}\left[\ddot{q}(t)-\ddot{q}_{d}(t)\right]dt=\left[\dtq(h)-\dtq_{d}(h)\right]-\left[\dtq(0)-\dtq_{d}(0)\right].
\ee
Further, observe that
\be
\begin{split}
&\int_{0}^{h}H_{n-1,j}(t)dt=\int_{0}^{h}H_{n-1,j}(h-t)dt=C_{j}h^{2},\quad j=1,\ldots n-2\\&
\int_{0}^{h}t^{n-1}(h-t)^{n-1}=Ch^{2n-1},
\end{split}
\ee
where $C_{j}, C>0$ are constants which do not depend on $h$. Now, let $\dtq_{d}(0)=\dtq(0)+\mathcal{O}(h^p)$ and
$\dtq_{d}(h)=\dtq(h)+\mathcal{O}(h^p)$, where we wish to determine the value of $p$. It is easy to see that after integrating both
sides of \eqref{lem:formula1} and rewriting it in terms of the order conditions we arrive at
$$
\mathcal{O}(h^p)=\mathcal{O}(h^{p+2})+\mathcal{O}(h^{2n-1}).
$$
Note that we used Lemma~\ref{lem:HOD} to estimate the higher-order derivatives in \eqref{lem:formula1}.
It follows immediately that $p=2n-1$ and the proof is finished.
\end{proof}

We are now ready to prove the main result of this section.
\begin{theorem}
Assume that a Lagrangian function $L:TQ\rightarrow \RR$ is sufficiently smooth and its partial derivatives
are uniformly bounded. Then, the discrete Lagrangian $L_{d}(q_{0},q_{1},h)$ constructed according to \eqref{DLEM}, \eqref{DLsys} with
$q_{0}=q(0)$ and $q_{1}=q(h)$, approximates the exact
discrete Lagrangian $L_{d}^{E}(q(0),q(h),h)$ with order $2\lfloor n/2\rfloor+3$, for $n\geq 3$. In particular,
when  $n$ is even, $\lfloor n/2\rfloor=n/2$, and the order is $n+3$. For odd $n$, $\lfloor n/2\rfloor=(n-1)/2$, and
the order is $n+2$. For $n=2$, the order is equal to $3$.
\end{theorem}
\begin{proof}
Observe that differentiability of the Lagrangian function $L$ implies that it is Lipschitz continuous in each of its arguments,
given that the partial derivatives are uniformly bounded. We will make use of both, differentiability and Lipschitz continuity
of $L$, in the proof below.

We start with the simplest nontrivial case of $n=2$, which corresponds to the space of piecewise cubic polynomials in the
Galerkin construction of a variational integrator. Note that it is sufficient to apply collocation to the Euler--Lagrange equation at the end-points
of the interval $[0,h]$ to uniquely define $q_{d}(t)$ satisfying the given boundary conditions. We obtain the order of approximation
of the first derivative by applying Lemma~\ref{lem:FOD}.
We use the simple trapezoidal rule
$$
\int_{a}^{b}f(t)dt=\frac{b-a}{2}(f(a)+f(b))+\mathcal{O}((b-a)^3).
$$
to discretize the action integral in \eqref{DL}. Since $L(q,\dtq)$ is Lipschitz continuous,
\ben
\label{lagr_est}
L(q_{d}(\tau),\dtq_{d}(\tau))=L(q(\tau),\dtq(\tau))+\mathcal{O}(h^{3}).
\een
Therefore,
\be
\begin{split}
L_{d}(q(0),q(h),h)&=\frac{h}{2}\left(L(q_{d}(0),\dtq_{d}(0))+L(q_{d}(h),\dtq_{d}(h))\right)\\&
=\frac{h}{2}\left(L(q(0),\dtq(0))+L(q(h),\dtq(h))\right)+\mathcal{O}(h^4)\\&
=L_{d}^{E}(q(0),q(h))+\mathcal{O}(h^3)+\mathcal{O}(h^4)\\&
=L_{d}^{E}(q(0),q(h))+\mathcal{O}(h^3).
\end{split}
\ee
The combination of the trapezoidal rule and the cubic Hermite interpolation makes the analysis of the approximation of the exact discrete Lagrangian by the discrete Lagrangian elementary. As we can see, the error
in approximation is determined by the error of the quadrature rule. This is due to the fact that the derivatives
are approximated to sufficiently high order. The pattern persists for higher-order Hermite interpolants
due to our choice of quadrature method.

It is straightforward to extend the above reasoning to the general case. Iteratively
applying Lemma~\ref{lem:HOD} and Lemma~\ref{lem:FOD} to the derivatives of $L_{d}(q_{d}(t),\dtq_{d}(t))$ with respect to time $t$ gives
the relation
\be
\frac{d^{2l-1}}{dt^{2l-1}}L(q_{d}(t),\dtq_{d}(t))\bigg|_{t=0,h}=\frac{d^{2l-1}}{dt^{2l-1}}L(q(t),\dtq(t))\bigg|_{t=0,h}+\mathcal{O}(h^{2n-1}),
\ee
which is analogous to \eqref{lagr_est}. Hence,
\be
\begin{split}
L_{d}(q(0),q(h),h)&=\frac{h}{2}(L(q_{d}(0),\dtq_{d}(0))+L(q_{d}(h),\dtq_{d}(h)))\\&
\quad-\sum_{l=1}^{\lfloor n/2\rfloor}\frac{B_{2l}}{(2l)!}h^{2l}\left(\frac{d^{2l-1}}{dt^{2l-1}}L(q_{d}(t),\dtq_{d}(t))\bigg|_{t=h}-
\frac{d^{2l-1}}{dt^{2l-1}}L(q_{d}(t),\dtq_{d}(t))\bigg|_{t=0}\right)\\&
=\frac{h}{2}(L((0),\dtq(0))+L(q(h),\dtq(h)))+\mathcal{O}(h^{2n})\\&
\quad-\sum_{l=1}^{\lfloor n/2\rfloor}\frac{B_{2l}}{(2l)!}h^{2l}\left(\frac{d^{2l-1}}{dt^{2l-1}}L(q(t),\dtq(t))\bigg|_{t=h}-
\frac{d^{2l-1}}{dt^{2l-1}}L(q(t),\dtq(t))\bigg|_{t=0}+\mathcal{O}(h^{2n-1})\right)\\&
=L_{d}^{E}(q(0),q(h),h)+\mathcal{O}(h^{2\lfloor n/2\rfloor+3})+\mbox{h.o.t.}
=L_{d}^{E}(q(0),q(h),h)+\mathcal{O}(h^{2\lfloor n/2\rfloor+3}).
\end{split}
\ee
\end{proof}

\textbf{Remarks.} Several remarks are in order. As we already noted,
the error in approximation of the exact discrete Lagrangian by the Prolongation-Collocation discrete Lagrangian is determined by the
order of the quadrature formula. In particular, the best order estimate is achieved if \emph{all} the collocation conditions \eqref{DLsys}
are used. The collocation equations enter the terms under the summation in \eqref{DLEM}, which in turn determine to the order of
accuracy of the quadrature formula.

Secondly, if we write $n=2k$, so that the degree of the interpolating polynomial is $d=2n-1=4k-1$, the choices
$n=2k$ and $n=2k+1$ lead to the same maximal order of the quadrature, $2k+3$. Therefore, it is preferable to use Hermite interpolating
polynomials of order $d=4k-1$, which minimizes the computational effort for a discrete Lagrangian for a given order of accuracy.
\begin{figure}[t!]
\begin{center}
\includegraphics[width=0.5\textwidth]{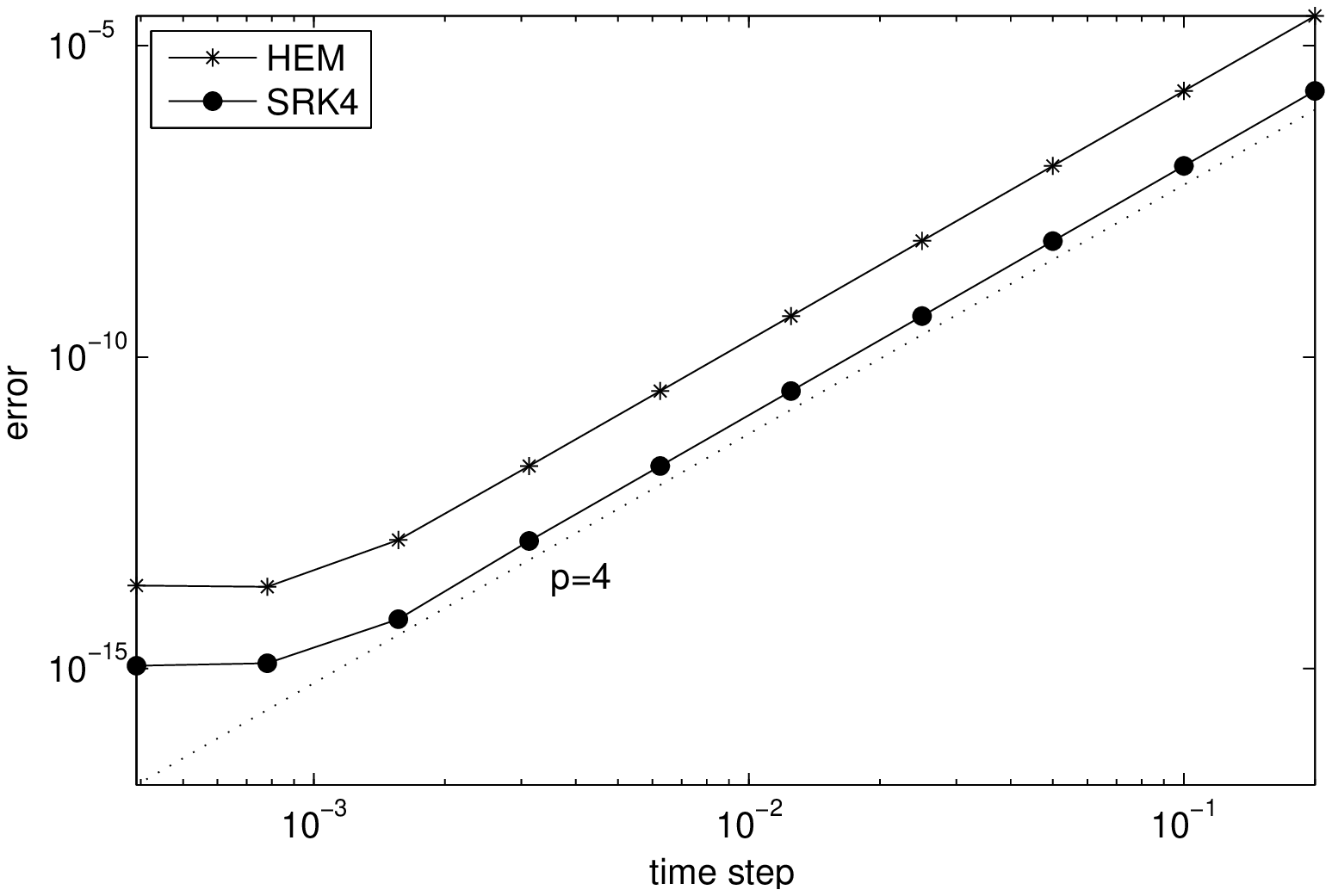}\includegraphics[width=0.5\textwidth]{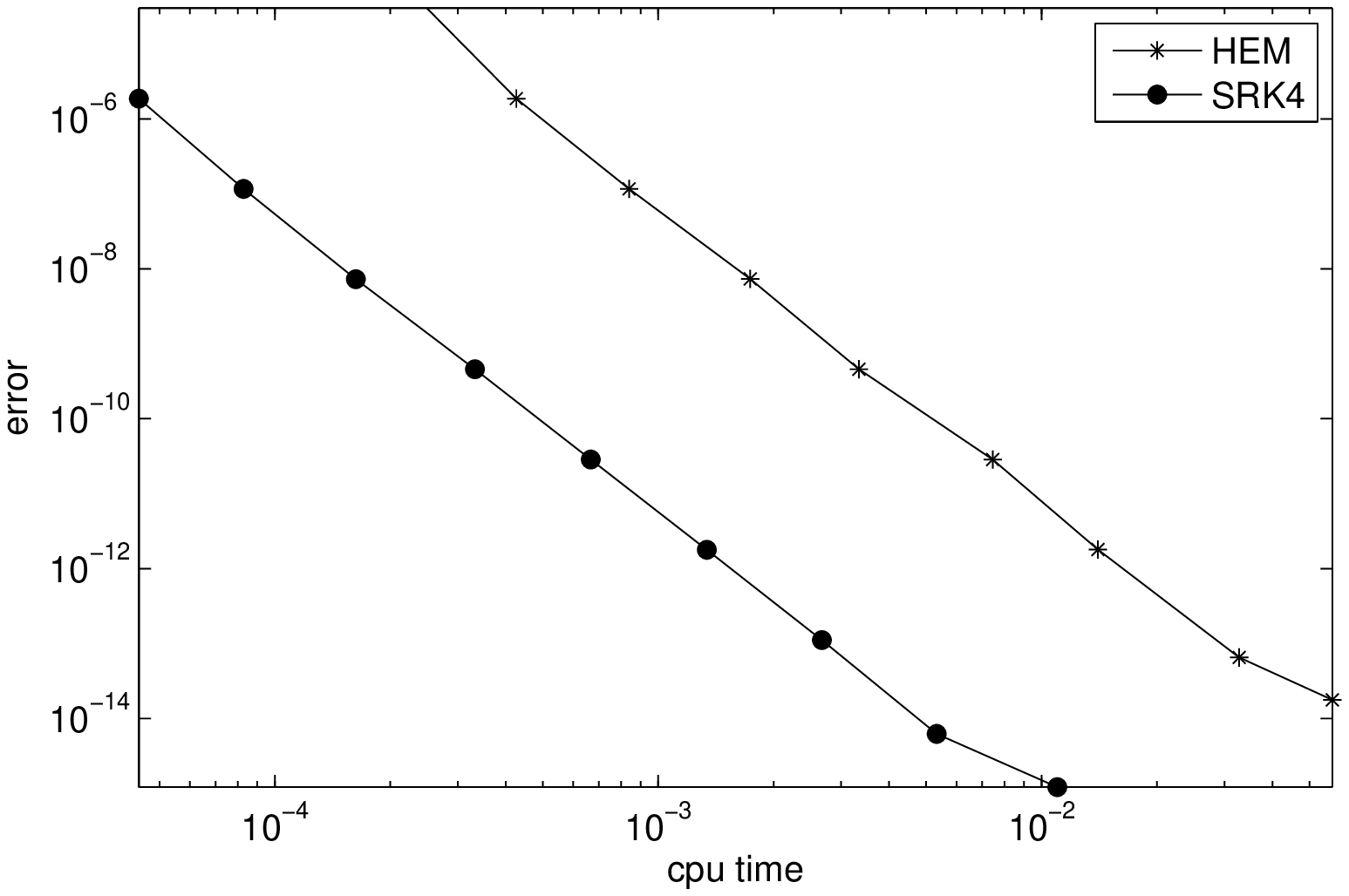}
\caption{Simple harmonic oscillator. Plots of: (a) the global errors for HEM and SRK4, (b) machine time versus the accuracy.
The dotted line is the reference line for the exact order.}
\label{fig:simpen_error&cputime}
\end{center}
\end{figure}

\begin{figure}[t!]
\begin{center}
\includegraphics[width=0.5\textwidth]{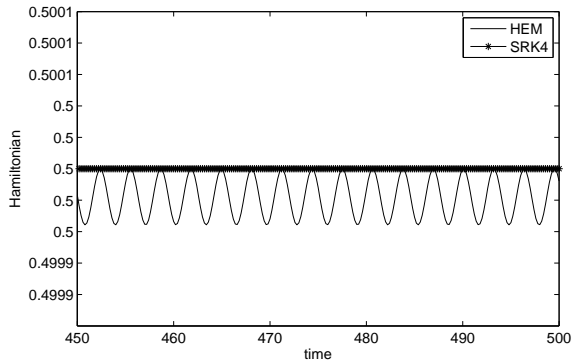}
\caption{Simple harmonic oscillator. Energy error for the $4$th order HEM and the two-stage symplectic RK4, step-size $h=0.2$.}
\label{fig:simpen_energy}
\end{center}
\end{figure}
\begin{figure}[t!]
\begin{center}
\includegraphics[width=0.5\textwidth]{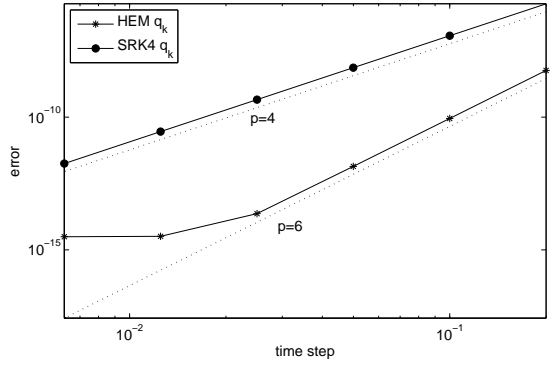}
\caption{Simple harmonic oscillator. The global approximation error of only the position trajectory.}
\label{fig:simpen_pos}
\end{center}
\end{figure}


\section{Examples}\label{section_examples}
\subsection{Simple Harmonic Oscillator}
We consider a harmonic oscillator system described by the equations
$$
\dtq=p,\qquad \dot{p}=-q.
$$
The total energy of the system is given by the Hamiltonian $H(q,p)=\frac{1}{2}p^2+\frac{1}{2}q^2$. To test our method numerically,
we used the $4$th order variational integrator (HEM) constructed by means of the quintic Hermite
interpolating polynomial and the Euler--Maclaurin quadrature formula. In the plots given in
Figure~\ref{fig:simpen_error&cputime}, Figure~\ref{fig:simpen_energy}, the resulting $4$th order method is compared to the
two-stage symplectic Runge--Kutta method of order 4.

We would like to note the following interesting fact. Let us consider only the position component $q_{k}$ of the
symplectic integrator $(q_{k},p_{k})\mapsto (q_{k+1},p_{k+1})$ and compute the order of the corresponding one-step method $q_{k}\mapsto q_{k+1}$.
It turns out that in the case of the 4th order HEM method applied to the simple harmonic oscillator, the global error in the position component is of order $6$.
The error plots are given in Figure~\ref{fig:simpen_pos}, where we can see that for the symplectic Runge-Kutta method the error
remains to be of the $4$th order. This does not however contradict the variational order analysis discussed in Section~\ref{section_VOCal}. The theorem mentioned
therein establishes the order of the integrator $(q_{k},p_{k})\mapsto (q_{k+1},p_{k+1})$ in position-momentum variables, but allows the integrator $(q_{k-1},q_k)\mapsto (q_k,q_{k+1})$ in position variables to have the same or higher order.

\subsection{Planar Pendulum}
A planar pendulum of mass $m=1$ with the massless rod of length $l=1$ is a Hamiltonian system for which the equations of motion
are
$$
\dtq=p,\qquad \dot{p}=-\sin q.
$$
In Figure~\ref{fig:genpen_error&cputime}, we compare the performance of the method proposed in Section~\ref{section_method} with the two-stage
symplectic Runge-Kutta method of the same order. HEM method encounters a slightly larger error in energy, but importantly it stays bounded
for large time-intervals.
\begin{figure}[t]
\begin{center}
\includegraphics[width=0.5\textwidth]{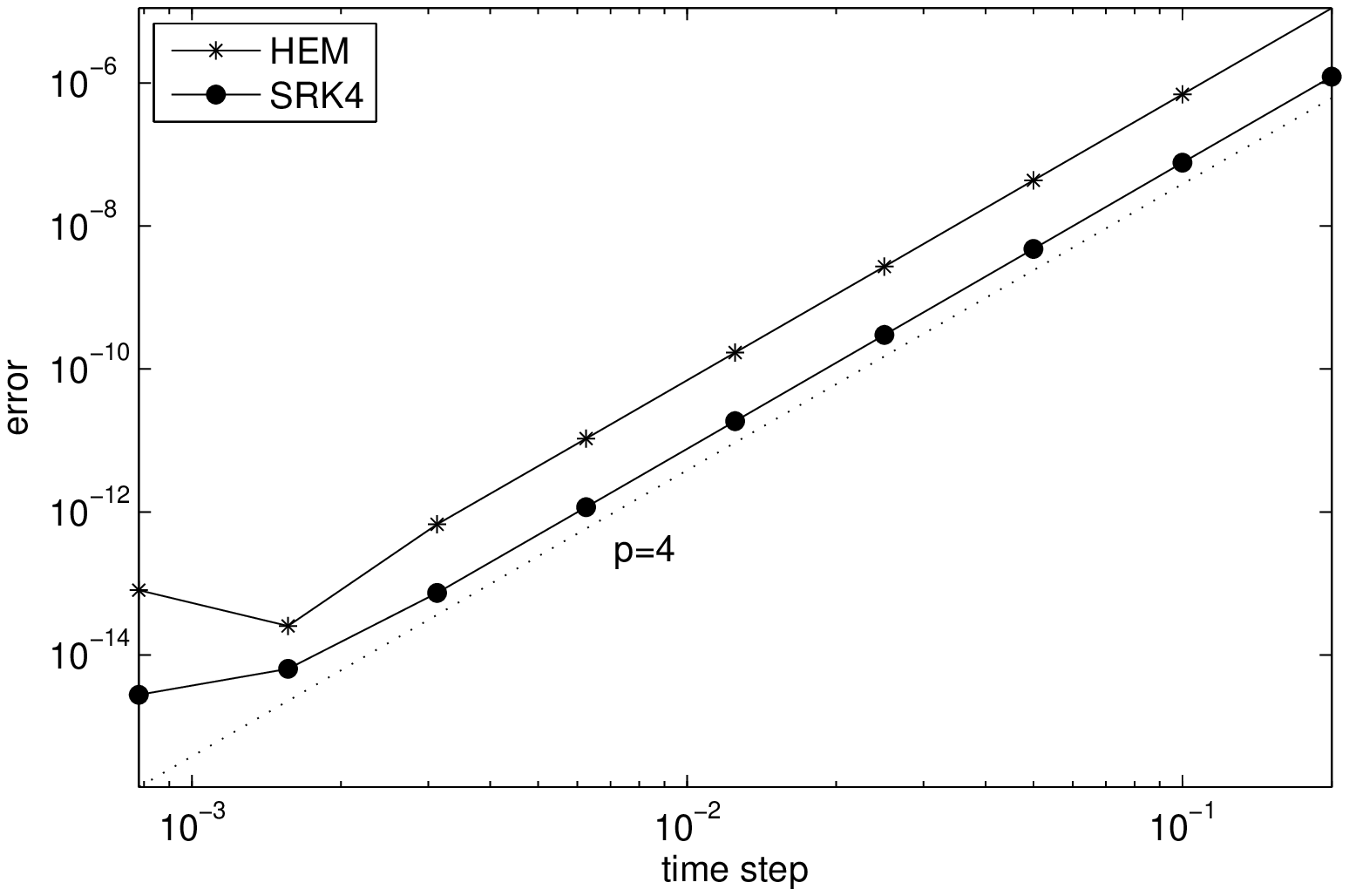}\includegraphics[width=0.5\textwidth]{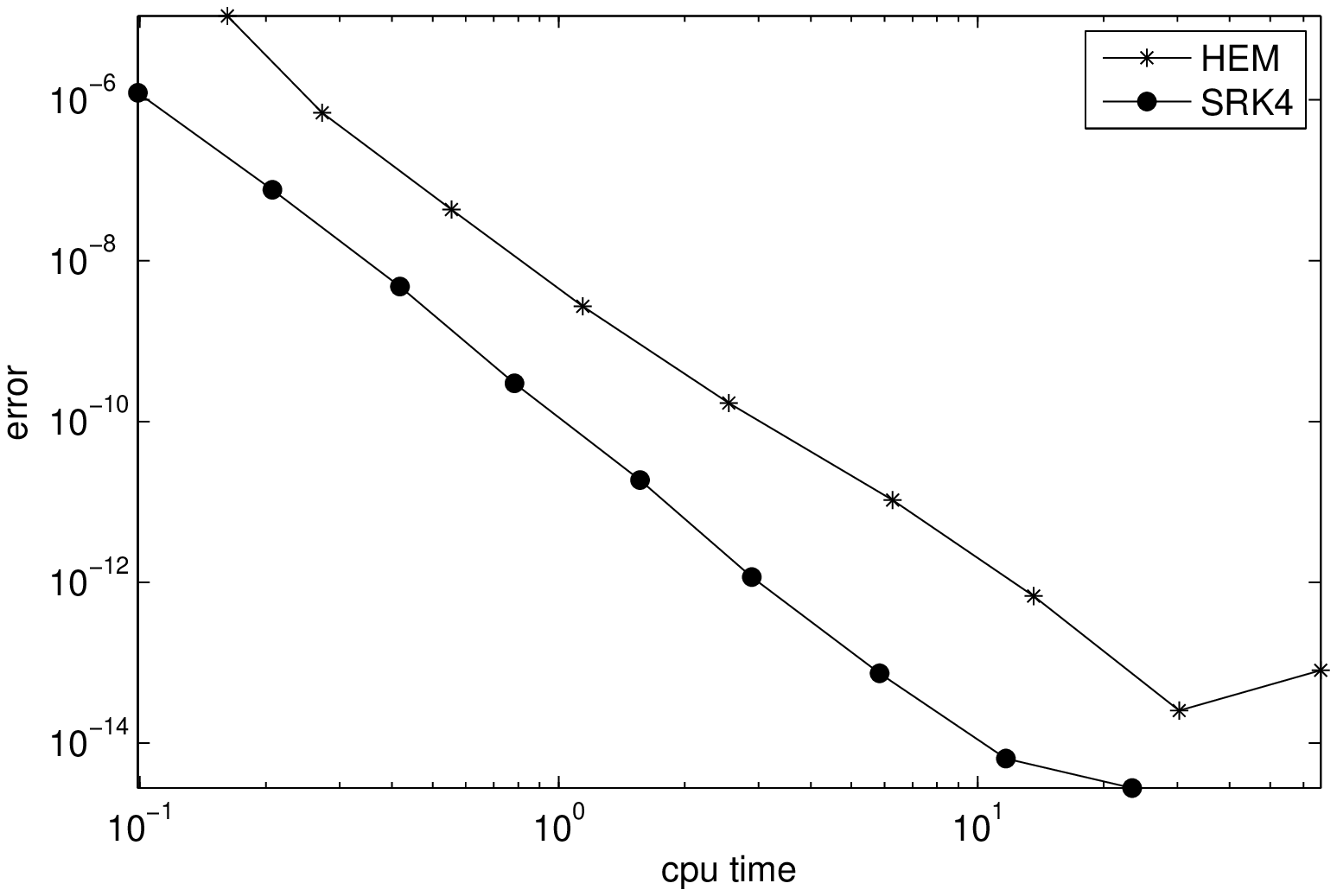}
\caption{Planar pendulum. Plots of: (a) the global errors for HEM and SRK4, (b) machine time versus the accuracy.
The dotted line is the reference line for the exact order.}
\label{fig:genpen_error&cputime}
\end{center}
\end{figure}

\begin{figure}[t]
\begin{center}
\includegraphics[width=0.5\textwidth]{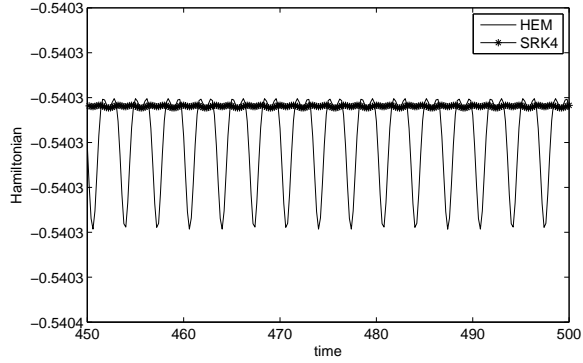}
\caption{Pendulum. Energy error for the $4$th order HEM and the two-stage symplectic RK4, step-size $h=0.2$.}
\label{fig:genpen_energy}
\end{center}
\end{figure}

\subsection{Duffing Oscillator}
The unforced undamped Duffing oscillator is a Hamiltonian system of equations
$$
\dtq=p,\qquad \dot{p}=q-q^3,
$$
where the Hamiltonian function is
$$
H(q,p)=\frac{1}{2}p^2-\frac{1}{2}q^2+\frac{1}{4}q^4.
$$
The plots in Figure~\ref{fig:duffing_error&cputime} show the comparison of the $2$nd order HEM variational integrator and the well-known
Midpoint rule. The Midpoint rule is an implicit integrator and HEM is semi-implicit meaning that $q_{k+1}$ satisfies an implicit equation
whereas $p_{k+1}$ is computed explicitly. The plots demonstrate the superiority of the HEM over the Midpoint rule in terms of the computational time.
The energy plots for both methods are given in Figure~\ref{fig:duffing_energy}.
\begin{figure}[t]
\begin{center}
\includegraphics[width=0.5\textwidth]{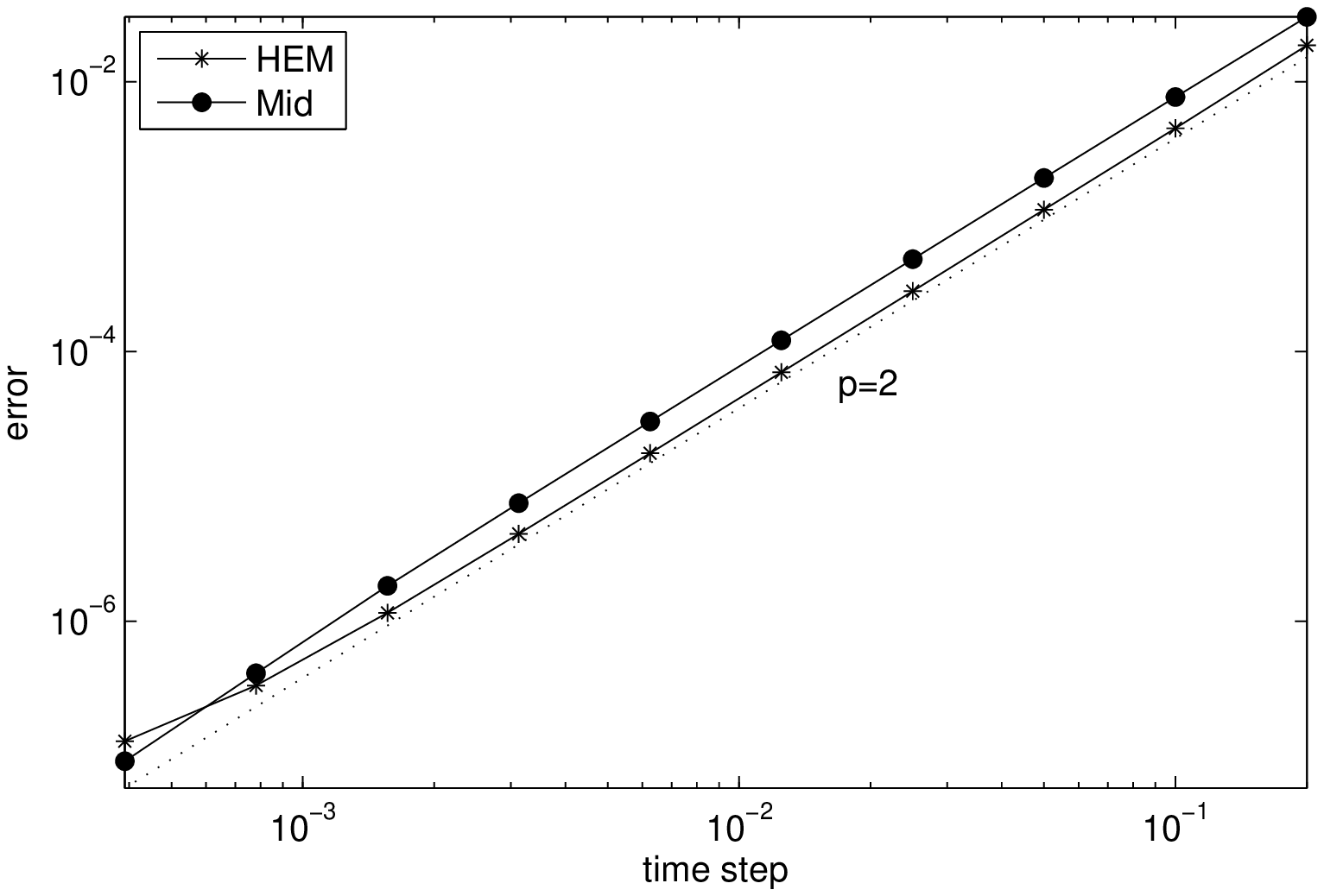}\includegraphics[width=0.5\textwidth]{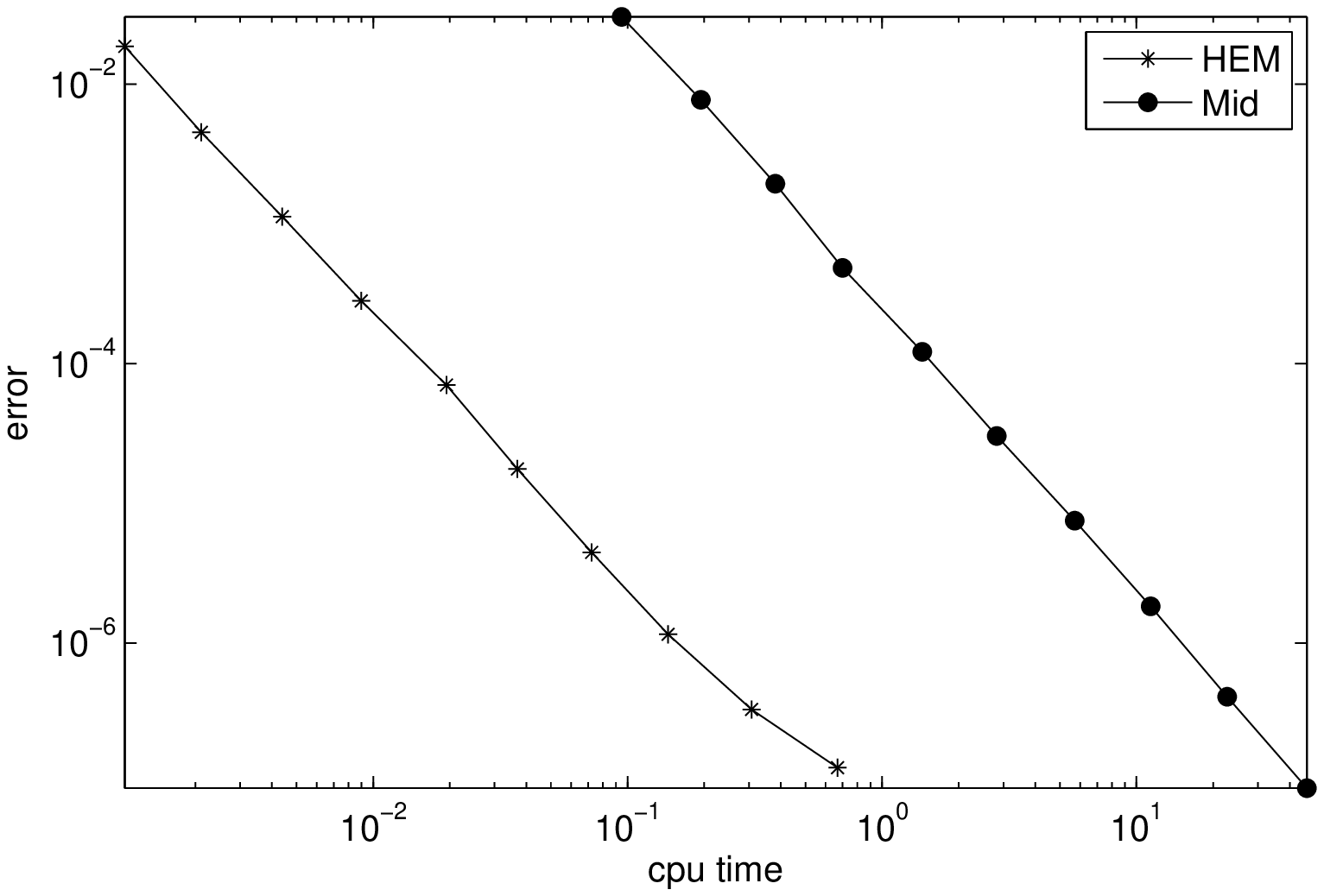}
\caption{Duffing oscillator. Plots of: (a) the global errors for HEM and the Midpoint rule, (b) machine time versus the accuracy.
The dotted line is the reference line for the exact order.}
\label{fig:duffing_error&cputime}
\end{center}
\end{figure}
\begin{figure}[t]
\begin{center}
\includegraphics[width=0.5\textwidth]{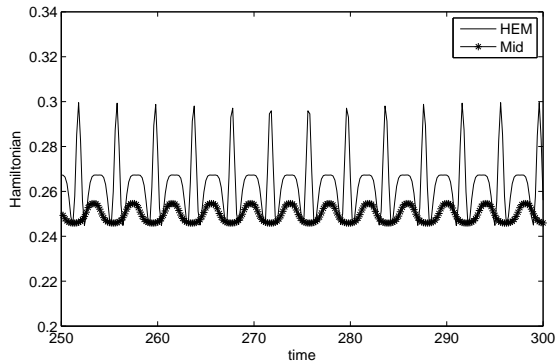}
\caption{Duffing oscillator. Energy error for the $2$nd order HEM and the Midpoint rule, step-size $h=0.2$.}
\label{fig:duffing_energy}
\end{center}
\end{figure}

\section{Conclusions and Future Directions}\label{sec:conclusions}
In this paper, we introduced a novel technique for constructing high-order variational integrators using collocation on the prolongation of the Euler--Lagrange equations. This relies on obtaining prolongation-collocation discrete curves with good approximation properties for the higher-derivatives. The resulting methods are particularly appropriate in combination with digital feedback control, since they naturally yield position and its derivatives as the output, without the need to use interpolation in order to access such data. This also naturally leads to numerical trajectories with better regularity properties across the time nodes.

It would be desirable to explore the connection between the methods proposed in this paper, and variational integrators based on global approximation techniques like splines, and to extend this work to the setting of Lie groups by incorporating techniques from Lie group variational integrators \cite{thesisLeok}, and constructive approximation techniques on Lie groups \cite{OS,S}. Furthermore, it would be interesting to extend the prolongation-collocation techniques to Hamiltonian variational integrators~\cite{LeZh2009} and Hamilton--Pontryagin variational integrators~\cite{LeOh2008} by considering prolongations of Hamilton's equations, and the implicit Euler--Lagrange equations.

\section*{Acknowledgements}
This research was partially supported by NSF grant DMS-1001521, and NSF CAREER Award DMS-1010687.
\bibliographystyle{plainnat}
\bibliography{referencesTS}

\begin{thebibliography}{16}
\providecommand{\natexlab}[1]{#1}
\providecommand{\url}[1]{\texttt{#1}}
\expandafter\ifx\csname urlstyle\endcsname\relax
  \providecommand{\doi}[1]{doi: #1}\else
  \providecommand{\doi}{doi: \begingroup \urlstyle{rm}\Url}\fi

\bibitem[Abramowitz and {I. A. Stegun (Eds)}(1972)]{abramowitz}
M.~Abramowitz and {I. A. Stegun (Eds)}.
\newblock \emph{Handbook of Mathematical Functions with Formulas, Graphs, and
  Mathematical Tables, 9th printing}.
\newblock New York, Dover, 1972.

\bibitem[Costabile and Napoli(2007)]{costabile}
F.~Costabile and A.~Napoli.
\newblock Solving {BVPs} using two-point {Taylor} formula by a symbolic
  software.
\newblock \emph{J. Comp. Appl. Math.}, 210:\penalty0 136--148, 2007.

\bibitem[Dal~Maso(1993)]{Da1993}
Gianni Dal~Maso.
\newblock \emph{An introduction to {$\Gamma$}-convergence}.
\newblock Progress in Nonlinear Differential Equations and their Applications,
  8. Birkh\"auser Boston Inc., Boston, MA, 1993.

\bibitem[Davis(1963)]{davis}
{P. J.} Davis.
\newblock \emph{Interpolation and Approximation}.
\newblock Blaisdell, New York, 1963.

\bibitem[Hairer et~al.(1993)Hairer, N{\o}rsett, and Wanner]{HaNoWa1993}
E.~Hairer, S.~P. N{\o}rsett, and G.~Wanner.
\newblock \emph{Solving ordinary differential equations. {I}}, volume~8 of
  \emph{Springer Series in Computational Mathematics}.
\newblock Springer-Verlag, Berlin, second edition, 1993.
\newblock Nonstiff problems.

\bibitem[Hairer et~al.(2006)Hairer, Lubich, and Wanner]{hairerbook}
E.~Hairer, C.~Lubich, and G.~Wanner.
\newblock \emph{Geometric Numerical Integration: Structure--Preserving
  Algorithms for Ordinary Differential Equations}.
\newblock Springer-Verlag, 2006.

\bibitem[Iserles(2009)]{iserlesbook}
A.~Iserles.
\newblock \emph{A First Course in the Numerical Analysis of Differential
  Equations}.
\newblock Cambridge University Press, 2009.

\bibitem[Leok(2004)]{thesisLeok}
M.~Leok.
\newblock \emph{Foundations of Computational Geometric Mechanics}.
\newblock PhD thesis, California Institute of Technology, 2004.

\bibitem[Leok and Ohsawa(2008)]{LeOh2008}
M.~Leok and T.~Ohsawa.
\newblock Variational discrete {D}irac mechanics--implicit discrete
  {L}agrangian and {H}amiltonian systems.
\newblock \emph{Foundations of Computational Mathematics}, 2008.
\newblock (submitted, {\tt arXiv:0810.0740 [math.SG]}).

\bibitem[Leok and Zhang(2010)]{LeZh2009}
M.~Leok and J.~Zhang.
\newblock Discrete {H}amiltonian variational integrators.
\newblock \emph{IMA Journal of Numerical Analysis}, 2010.
\newblock (accepted, {\tt arXiv:1001.1408 [math.NA]}).

\bibitem[Marsden and West(2001)]{marsden&west}
{J. E.} Marsden and M.~West.
\newblock Discrete mechanics and variational integrators.
\newblock \emph{Acta Numerica}, 10:\penalty0 357--514, 2001.

\bibitem[Olver(1993)]{Ol1993}
Peter~J. Olver.
\newblock \emph{Applications of {L}ie groups to differential equations}, volume
  107 of \emph{Graduate Texts in Mathematics}.
\newblock Springer-Verlag, New York, second edition, 1993.

\bibitem[Oswald and Shingel(2009)]{OS}
P.~Oswald and T.~Shingel.
\newblock Splitting methods for $\mathrm{SU}(\mathrm{N})$ loop approximation.
\newblock \emph{J. Approx. Th.}, 161(1):\penalty0 174--186, 2009.

\bibitem[Patrick and Cuell(2009)]{PaCu2009}
George~W. Patrick and Charles Cuell.
\newblock Error analysis of variational integrators of unconstrained
  {L}agrangian systems.
\newblock \emph{Numer. Math.}, 113\penalty0 (2):\penalty0 243--264, 2009.

\bibitem[Rall(1981)]{rall1981}
Louis~B. Rall.
\newblock \emph{Automatic Differentiation: Techniques and Applications}, volume
  120 of \emph{Lecture Notes in Computer Science}.
\newblock Springer, 1981.

\bibitem[Shingel(2010)]{S}
T.~Shingel.
\newblock Trigonometric approximation of $\mathrm{SO}(\mathrm{N})$ loops.
\newblock \emph{Constr. Approx.}, 32(3):\penalty0 597--618, 2010.

\end{thebibliography}

\end{document}